\documentclass[10pt]{amsart}
\usepackage{geometry} 
\usepackage{amssymb}
\usepackage{amscd}
\usepackage{amsfonts}
\usepackage{amsmath}
\usepackage{amssymb}
\usepackage[american, english]{babel}
\usepackage{bbm}
\usepackage{bookmark}
\usepackage{cmap}
\usepackage{dsfont}
\usepackage{enumerate}
\usepackage{epigraph}
\usepackage{euscript}
\usepackage[myheadings]{fullpage}
\usepackage{graphicx}
\usepackage{mathabx}
\usepackage{mathrsfs}
\usepackage{mathtools}
\usepackage{refcount}
\usepackage{stmaryrd}
\usepackage{thinsp}
\usepackage[safe]{tipa}
\usepackage{verbatim}
\usepackage{xr-hyper}
\usepackage{hyperref}
\usepackage[all]{xy}
\usepackage{url}
\usepackage{geometry} 
\usepackage{hyperref}
\usepackage{amsmath}
\newtheorem{defi}{Definition}[subsection]
\newtheorem*{defi*}{Definition}
\newtheorem{rmk}[defi]{Remark}
\newtheorem*{rmk*}{Remark}
\newtheorem*{prop*}{Proposition}
\newtheorem{thm}[defi]{Theorem}
\newtheorem*{thm*}{Theorem}
\newtheorem*{lmm*}{Lemma}
\newtheorem{lmm}[defi]{Lemma}

\DeclareSymbolFont{largesymbols}{OMX}{yhex}{m}{n}
\DeclareMathAccent{\wideparen}{\mathord}{largesymbols}{"F3}

\title{On centralizers in Azumaya domains}
\author{Thomas Bitoun} \email{Thomas.Bitoun@ucalgary.ca}
\author{Justin Desrochers} \email{Justin.Desrochers@ucalgary.ca}

\begin{document}

\newtheorem{MainThm}{Theorem} 
\renewcommand{\theMainThm}{\Alph{MainThm}}. 

\maketitle

\begin{abstract} We prove a positive characteristic analogue of the classical result that the centralizer of a nonconstant differential operator in one variable is commutative. This leads to a new, short proof of that classical characteristic zero result, by reduction modulo $p.$\end{abstract}


\section{Introduction}

The theory of commuting differential operators in one variable goes back at least a century \cite{MR1575373}. Of special interest is the following fundamental result:

\begin{thm*} Let 
$P= a_n\frac{d^n}{dx}+ \dots + a_1 \frac{d}{dx}+a_0, a_i\in \mathbb{C}[x],$ be a differential operator of positive degree. Then the algebra of differential operators that commute with $P$ is commutative.
\end{thm*}

The modern version of this result is attributed to Flanders in \cite{MR95305} and has been given at least two rather different proofs \cite{MR95305, MR2083753}.
In this note, we are interested in an analogous statement in positive characteristic. Namely for the first Weyl algebra $A_1(k)$ over a field $k$ of positive characteristic. We prove a generalization of the following:

\begin{thm*}

Let $k$ be a field of positive characteristic and let $P\in A_1(k)$ be noncentral. Then the centralizer of $P$ is a commutative algebra.
\end{thm*}

We also prove that the fraction field of the centralizer is isomorphic to the fraction field of $Z[P],$ where $Z$ is the center of $A_1(k).$ Note that this is simpler than the analogous assertion for a complex differential operator $Q,$ for which the fraction field of the centralizer is in general only a finite extension of the field of rational functions in $Q,$ see \cite[Corollary 1]{MR95305}.
 Unfortunately, neither the elementary methods from \cite{MR95305} (because of operators of degree divisible by the characteristic), nor those from \cite{MR2083753} (because all operators are algebraic over the center, which is of dimension $2$) do adapt to the positive characteristic. Nevertheless, we present here a simple argument based on a dimension count. 
We note finally that, in the spirit of \cite{MR794737} and \cite{MR2337879}, this provides a new and very short proof of the classical characteristic zero theorem above, by reduction modulo primes.   

\subsection*{Acknowledgements}

We are grateful to Konstantin Ardakov for suggesting a simplification of our initial argument. We also thank the referees for their comments which have helped improve the exposition.
This work was supported by the Natural Sciences and Engineering Research Council of Canada (NSERC), [RGPIN-2020-06075].
Cette recherche a été financée par le Conseil de recherches en sciences naturelles et en génie du Canada (CRSNG), [RGPIN-2020-06075].

\section{Main result}

Here is the main result of this note.

\begin{lmm} Let $A$ be a domain that is a finitely generated module of rank $p^2$ over its center $Z,$ for $p$ a prime number. Then the centralizer $B_a$ of every noncentral element $a$ is a commutative ring. Moreover, the natural embedding $Z[a] \subseteq B_a$ induces an isomorphism of fields of fractions. 
\end{lmm}

\begin{proof} Let $a\in A$ be an element which does not belong to the center $Z,$ and let $B=B_a$ be the centralizer of $a$ in $A.$ We let $K$ be the fraction field of $Z, L$ be the fraction field of the center $R$ of $B,$ and $L'$ be the fraction field of $Z[a].$ Note that $L' \subseteq L.$

Consider $A_K:= K\otimes_Z A, B_L:= L\otimes_R B,$ and $B_{L'}:= L'\otimes_{Z[a]}B.$ Since these are all localizations of domains, they are domains. Note that $A_K$ is of finite dimension $p^2$ over its central subfield $K.$ Hence it is a division ring, as multiplication by a nonzero element is injective and hence invertible, by the finite dimension. Thus there are natural inclusions of the localizations $B_{L'}\subseteq B_L\subseteq A_K,$ and $B_L$ and $B_{L'}$ are also division rings, since they are finite dimensional over their central subfield $K.$ 
Moreover the dimension $\mathrm{dim}_K(B_L)$ (resp. $\mathrm{dim}_K(B_{L'})$) divides $\mathrm{dim}_K(A_K)=p^2.$ Thus $\mathrm{dim}_K(B_L)$ (resp. $\mathrm{dim}_K(B_{L'})$) is $1, p$ or $p^2.$ But the inclusions $K\subsetneq B_L \subsetneq A_K$ (resp. $K\subsetneq B_{L'} \subsetneq A_K$) are proper, since $a$ is a noncentral element. Hence  $\mathrm{dim}_K(B_L)=p= \mathrm{dim}_K(B_{L'}).$ Thus $B_{L'}=B_L.$

Finally, we have that $\mathrm{dim}_L(B_L)$ (resp. $\mathrm{dim}_{L'}(B_{L'})$) divides $\mathrm{dim}_K(B_L)=p.$ Thus $\mathrm{dim}_L(B_L)$ (resp. $\mathrm{dim}_{L'}(B_{L'})$) is either $1$ or $p.$ But the inclusion $K\subsetneq L$ (resp. $K\subsetneq L'$) is proper as $a$ is not central, hence $\mathrm{dim}_L(B_L) < \mathrm{dim}_K(B_L)$ (resp. $\mathrm{dim}_{L'}(B_{L'}) < \mathrm{dim}_K(B_{L'})$). We conclude that $\mathrm{dim}_L(B_L)=1= \mathrm{dim}_{L'}(B_{L'}),$ thus $B\subseteq L$ is commutative and the fraction field of $B$ is $B_L= B_{L'}=L'.$ \end{proof}

\begin{rmk} If $A$ is as in the lemma and is also an algebra over a field $k,$ then $k[a]$ is included in $B_a.$ However $B_a$ is not necessarily finite over $k[a].$ For example, in the first Weyl algebra over a field $k$ of positive characteristic $p$ with coordinate $x,$ we have $B_x= k[x, (\frac{d}{dx})^p].$ This algebra is not a finitely generated module over $k[x].$ \end{rmk}

\begin{rmk} In the case of a domain of higher rank $p^n$ over its center, the centralizer of a noncentral element is not necessarily commutative. For example, the centralizer of $x$ in the second Weyl algebra $A_2(k),$ with coordinates $x$ and $y,$ is generated over the positive characteristic $p$ field $k$ by $x, \partial_x^p, y$ and $\partial_y,$ and is thus not commutative. Nevertheless, we believe that the rank of a centralizer over the center is a useful invariant here too, and hope to consider it in a future work. \end{rmk}

\section{A corollary} 

The lemma applies in particular to the first Weyl algebra over a field of positive characteristic. This leads to a proof by reduction modulo $p$ of the following classical result \cite{MR95305, MR2083753}.

\begin{thm} Let $k$ be a field of characteristic zero. Then the centralizer of every nonconstant polynomial differential operator in one variable over $k$ is a commutative ring.    
\end{thm}

\begin{proof}
For an arbitrary commutative ring $k',$ we denote by $A(k')$ the ring of polynomial differential operators in one variable over $k',$ i.e. the first Weyl algebra over $k'.$ Moreover, for all $ a \in A(k'),$ we let the total degree $\mathrm{tot}(a)$ of $a$ be the degree of the total symbol of $a$ as a $k'$-polynomial in $2$ variables.  

Let $a\in A(k)$ be nonconstant and let $P, Q\in A(k)$ be operators commuting with $a.$ We want to show that the commutator $C:= [P,Q]$ vanishes. We let $S$ be the ring generated by the coefficients of $a, P,$ and $Q,$ lying in $k.$ We have $a, P, Q, C \in A(S)\subseteq A(k).$  

Let $n$ be the total degree of $a= \Sigma_{i+j\leq n} a_{ij} x^i\frac{d^j}{dx},$ and let $u= N \times \Pi_{i+j=n}a_{ij},$ where $N$ is the factorial of $n$ and the product $\Pi$ is taken over nonzero coefficients only. We note that the ring $S[\frac{1}{u}]$ is Jacobson \cite[Cor.10.4.6]{MR217086}. Hence the closed points of $\mathrm{Spec}(S[\frac{1}{u}])$ form a dense subset, and for each closed point $s,$ the residue field $k(s)$ is a finite field. Let us denote by $a_s, P_s, Q_s,$ and $C_s$ the images of $a, P, Q,$ and $C$ in $A(k(s)),$ respectively. Then by the choice of $u,$ for all closed point $s$ of $\mathrm{Spec}(S[\frac{1}{u}]),$ the element $a_s$ is of positive total degree which is prime to the characteristic of $k(s).$ Hence $a_s$ is not central, $P_s$ and $Q_s$ commute with $a_s,$ and $C_s=[P_s,Q_s].$ 

Since the ring $A(k(s)$) is a domain of rank $p_s^2$ over its center \cite[Thm. 2 and \S 4]{MR335564}, for $p_s$ the characteristic of $k(s),$ we can apply the lemma. Hence we conclude that $C_s=0$ for all closed points $s$ of $\mathrm{Spec}(S[\frac{1}{u}]).$ This holds generically since the closed points of $\mathrm{Spec}(S[\frac{1}{u}])$ are dense. Thus $C=0.$ \end{proof}

\bibliography{bibfilex}
\bibliographystyle{plain}

\end{document}